\documentclass[11pt,letterpaper]{amsart}
\usepackage{latexsym,amsfonts,amssymb,amsmath,amsthm, datetime2}
\usepackage{graphicx}
\usepackage{color}
\usepackage{mathrsfs}
\usepackage{amsmath, amsthm, amsfonts, amssymb}
\usepackage{wrapfig}
\usepackage{stackrel}
\usepackage{color}
\usepackage{float}
\usepackage{enumitem}
\usepackage{mathtools}
\usepackage{multicol}
\usepackage[normalem]{ulem}
\usepackage{xcolor}
\usepackage{hyperref}
\usepackage{cancel}

\usepackage[margin=1in]{geometry}

\newcommand{\mz}{\ensuremath{\mathbb Z}}

\newcommand{\mq}{\ensuremath{\mathbb Q}}
\newcommand{\mc}{\ensuremath{\mathbb C}}

\newcommand{\mcs}{\ensuremath{\mathcal S}}
\newcommand{\mymod}{\ensuremath{\negthickspace \negmedspace \pmod}}
\newcommand{\shortmod}{\ensuremath{\negthickspace \negthickspace \negthickspace \pmod}}

\newcommand{\half}{\ensuremath{ \frac{1}{2}}}

\newcommand{\chibar}{\overline{\chi}}
\newcommand{\thalf}{\tfrac12}
\newcommand{\sumstar}{\sideset{}{^*}\sum}

\DeclarePairedDelimiter\autobracket{(}{)}
\newcommand{\pb}[1]{\autobracket*{#1}}

\newcommand{\ov}{\overline}

\newcommand{\sumn}{\sum_{n \geq 1}}
\newcommand{\sumnq}{\sum_{\substack{n\ge1 \\ (n,q)=1}}}

\newcommand{\sumstarchiq}{\ \sideset{}{^*}\sum_{\chi \shortmod{q}}}

\newcommand{\sumstarchipsplusminus}{\ \sideset{}{^*}\sum_{\substack{\chi \mymod{q}\\ \chi(-1)=\pm 1}}}

\newcommand{\sgn}{\mathrm{sgn}}

\newtheorem{prop}{Proposition}

\theoremstyle{plain}		
	\newtheorem{mytheo}{Theorem}[section]

	\newtheorem{mycoro}[mytheo]{Corollary}
     \newtheorem{mylemma}[mytheo]{Lemma}

	\newtheorem{myremark}[mytheo]{Remark}

\theoremstyle{remark}

\numberwithin{equation}{section}

\begin{document}
\author{Agniva Dasgupta}
\author{Rizwanur Khan}
\author{Ze Sen Tang}

 \address{Department of Mathematical Sciences\\ University of Texas at Dallas\\ Richardson, TX 75080-3021}
 \email{agniva.dasgupta@utdallas.edu, rizwanur.khan@utdallas.edu, zesen.tang@utdallas.edu}

  \thanks{This work was supported by the National Science Foundation grants DMS-2341239 and DMS-2344044.}

% \begin{abstract}
% \end{abstract}
\title{Reciprocity for $GL(2)$ $L$-functions twisted by Dirichlet characters}

\begin{abstract}
A formula connecting a moment of $L$-functions and a dual moment in a way that interchanges the roles of certain key parameters on both sides is known as a reciprocity relation. We establish a reciprocity relation for a first moment of $GL(2)$ $L$-functions twisted by Dirichlet characters. This extends, via a new and simple argument, some results of Bettin, Drappeau, and Nordentoft.

\end{abstract}

\maketitle

\section{Introduction}

\subsection{Statement of results}
\label{subsec:mainresult}

Reciprocity relations lie at the heart of many fundamental results in number theory. One of the earliest instances of this phenomenon is the celebrated law of quadratic reciprocity, which relates the Legendre symbols $(\frac{p}{q})$ and $(\frac{q}{p})$. In the context of $L$-functions, reciprocity refers to a relation between a moment of $L$-functions and a dual moment in a way that swaps
the roles of certain key parameters on both sides. A beautiful example of this phenomenon is the following formula discovered by Conrey \cite{conrey}, and later improved by Young \cite{young}, concerning the twisted second moment of Dirichlet $L$-functions. For $p$ and $q$ distinct odd primes, let $c(q)=2q^{\frac12}-2$ and define the twisted moment 
\begin{equation}
\label{eq:twistmomdeg1}
M(p,q) = \frac{1}{\phi(q)} \pb{\sumstarchiq q^{\frac12}\left\lvert L\pb{\frac12,\chi} \right\rvert^2 \chi(p) +  c(q) \zeta\pb{\half}^2},
\end{equation}
where $\sum^*$ indicates that the sum is over primitive characters. We have
\begin{prop}
\label{prop:con}\cite[Corollary 1.2]{young}
For any fixed $\varepsilon, C>0$ and odd primes $p<q^{1-\varepsilon}$, we have
\begin{equation}
\label{eq:con}
    M(p,q) - M(-q,p) = \zeta\pb{\half}^2 + \frac{q^{\frac12}}{p^{\half}} \pb{\log\frac{q}{p} + \gamma -\log 8\pi} +  O(pq^{-1+\varepsilon} + p^{-C}).
\end{equation}
\end{prop}
\noindent The point of this formula is that it relates the moment $ M(p,q) $ to the dual moment $ M(-q,p) $, so that the key parameters $p$ and $q$ are essentially swapped. One corollary of \eqref{eq:con} is an asymptotic for $M(p,q)-M(-q,p)$ for $p< q^{1-\varepsilon}$, as $q\to \infty$, while an asymptotic for $M(p,q)$ alone is only known for twists of size $p< q^{\half-\varepsilon}$. The result above was extended by Bettin \cite[Theorem 3]{bettin}, who obtained a full asymptotic series for $M(p,q)-M(-q,p)$, up to an arbitrarily small error term, under the same condition $p<q^{1-\varepsilon}$. Bettin also obtained a reciprocity relation for this second moment with a more general twist, but only up to a modest error term. To state his result, we define the twisted moment
\begin{equation}
\label{eq:twistmomdeg1}
M_{\pm}(p,r;q) = \frac{1}{\phi(q)} \sumstar_{\substack{\chi \shortmod{q}\\ \chi(-1)=\pm 1}} q^{\frac12}\left\lvert L\pb{\frac12,\chi} \right\rvert^2 \chi(p) \chibar({r}).
\end{equation}
Here and throughout the rest of the paper, we will assume that $p,q,r$ are distinct odd primes.
\begin{prop}{\cite[Corollary 3]{bettin}}
\label{prop:bet}
For distinct odd primes $p,q,r$ with $4pr \leq q$ we have,
\begin{equation}
\label{eq:con}
    M_{\pm}(p,r;q)  \mp M_{\pm}(p,q;r)  \mp M_{\pm}(r,q;p) = \frac12 \pb{\frac{q}{pr}}^{\frac12} \pb{\log\frac{q}{pr} + \gamma -\log 8\pi \mp \frac{\pi}{2}} +  O(\log{q}).
\end{equation}
\end{prop}
\noindent It is interesting that this relation involves three twisted moments.

We now turn to cuspidal analogues of Propositions \ref{prop:con} and \ref{prop:bet}. In the first situation, let $f$ be a holomorphic Hecke cusp form of weight $k\ge 12$ for $SL_2(\mathbb{Z})$. Let $\lambda_f(n)$ denote the Hecke eigenvalues associated to $f$, normalized so that the analytical continuation of $L(s,f)= \sum_{n\ge 1} \frac{\lambda_f(n)}{n^s}$ has critical line $\Re(s)=\frac12$. Let
\[
c_f(q) = \lambda_f(q)q^{\frac12}-2
\]
and define the twisted moment
    \begin{equation}
    \label{eq:Mfdefn}
        M_f(p,r;q) = \frac{1}{\phi(q)} \pb{ \sumstarchiq \tau(\chi) L\pb{\half, f \otimes \chibar} \chi(p) \chibar(r) +c_f(q) L\pb{\half, f}},
    \end{equation}
    where $\tau(\chi)$ denotes the Gauss sum. Note that $M_f(p,r;q)$ is the cuspidal analogue of 
    \[
    M(p,r;q)=M_+(p,r;q)+ i M_-(p,r;q)
    \]
     since $i^\eta q^\half |L(\half, \chi)|^2=\tau(\chi)L(\half, \overline{\chi})^2$, where $\eta=0$ when $\chi$ is even, and $\eta=1$ when $\chi$ is odd. 
     
     For $\alpha \in \mathbb{Q}$, define the additively twisted $L$-series associated to $f$ by
\begin{equation}
\label{eq:addtwistdef}
    L(s, f ,\alpha)= \sum_{n \geq 1}\frac{\lambda_f(n)e(n\alpha)}{n^s}
\end{equation}
for $\Re(s)>1$, which analytically continues to the rest of the complex plane with a functional equation \cite[Section 3.3]{nord2}. The central values of these $L$-series are a convenient way to package the twisted moments that we are interested in, since as we will see in Lemma \ref{lem:addtwistmoment}, we have 
\[
M_f(p,r;q) = L\pb{\half, f , \frac{\ov{p}r}{q}},
\]
where $\ov{p}$ denotes any integer such that $\ov{p} p \equiv 1 \bmod q$. This quantity is related to the concept of ``modular symbols'' (see \eqref{period} and \cite{man}). Our first main theorem is the following.
\begin{mytheo}
\label{thm:mainthm} 
   Let $f$ be a fixed holomorphic Hecke cusp form of weight $k\ge 12$ for $SL_2(\mathbb{Z})$. Let $p,q,r$ be distinct odd primes. We have
           \begin{multline}
         \label{eq:thmhol}
            M_f(p,r;q)- M_f(-q,r;p) -  M_f(-p,q;r) =\\ \sum_{j=1}^{\frac{k}{2}-1} \frac{1}{j!} \frac{\Gamma(\frac{k}{2}+j)}{\Gamma(\frac{k}{2}-j)} \pb{\frac{2\pi i q}{pr}}^{-j} \pb{L\pb{\frac12 +j, f , -\frac{\ov{p}q}{r}} + (-1)^j i^k  L\pb{\frac12 +j, f , \frac{q\ov{r}}{p}}}.  
         \end{multline} 
The $L$-series values on the right hand side are in the region of absolute convergence of $L(s, f, \alpha)$. If $pr<Cq$ for a constant $C$, then $ M_f(p,r;q)- M_f(-q,r;p) - M_f(-p,q;r) = O(\frac{pr}{q})$. 
\end{mytheo}
\noindent Note that this reciprocity relation is an exact formula (there is no error term) involving only finitely many terms. Although we assumed $r$ to be prime, we can extend the arguments in the proof of Theorem \ref{thm:mainthm} to the case $r=1$ by defining $ M_f(p,q;1) = L(f,\thalf)$ (see Lemma \ref{lem:addtwistmoment}). We thus obtain, as a corollary, a reciprocity relation with a simple twist that was previously established by Bettin and Drappeau \cite{bettdrap} and Nordentoft \cite{nord} using a different proof. To state this, define the twisted moment
    \begin{equation}
        M_f(p,q) = \frac{1}{\phi(q)} \pb{ \sumstarchiq \tau(\chi) L\pb{\half, f \otimes \chibar} \chi(p)  +c_f(q) L\pb{\half, f}}.
            \end{equation}
\begin{mycoro}
\label{maincor} 
We have
        \begin{multline}
         \label{eq:corhol}
        M_f(p,q)- M_f(-q,p) = L\pb{\half, f}  \ + \\ \sum_{j=1}^{\frac{k}{2}-1} \frac{1}{j!} \frac{\Gamma(\frac{k}{2}+j)}{\Gamma(\frac{k}{2}-j)} \pb{\frac{2\pi i q}{p}}^{-j} \pb{L\pb{\frac12+j, f} + (-1)^{j} i^k  L\pb{\half+j, f ,\frac{q}{p}}}.
        \end{multline}
\end{mycoro}

Now consider another situation: suppose that $g$ is a Hecke Maass cusp form for $SL_2(\mathbb{Z})$ and let $c_g(q) = \lambda_g(q)q^\half-2$ as before. We define the twisted moments
\begin{equation}
    \label{eq:Mgdefn+}
        M_g^{\pm}(p,r;q) = \frac{1}{\phi(q)} \pb{ \sumstarchipsplusminus \tau(\chi) L\pb{\half, g \otimes \chibar} \chi(p) \chibar(r) + (1-\eta_{\pm}) c_g(q) L\pb{\half, g}},
\end{equation}
where $\eta_{+} = 0$ and $\eta_{-} =1$. 
Our reciprocity relation in this case is an asymptotic series with an arbitrarily small error term.
\begin{mytheo}
    \label{thm:mainthm2}
Let $g$ be a fixed Hecke Maass cusp form for $SL_2(\mathbb{Z})$. Let $p,q,r$ be distinct odd primes with $pr<q^{1-\varepsilon}$. For any $M \in \mz_{\geq 1}$, we have 
    \begin{equation}
    \label{eq:thm2eq}      M_g^{\pm}(p,r;q) \mp  M_g^{\pm}(q,r;p) \mp  M_g^{\pm}(p,q;r) = \sum_{m=1}^{M} b^\pm_{m,g}\pb{\frac{2 \pi  q}{pr}}^{-m}+ O_M\pb{\pb{\frac{ q}{pr}}^{-M-1}},
     \end{equation}
where $b^\pm_{m,g}$ are constants depending only on $m$ and $g$. 
\end{mytheo}
\noindent Once again, we can extend these results to the case $r=1$ by defining $ M_g^{\pm}(p,q;1)=(1-\eta_\pm)  L(\half, g)$. This gives us the following corollary for a simple twist, for which we define
    \begin{equation}
        M_g^{\pm}(p,q) = \frac{1}{\phi(q)} \pb{ \sumstarchipsplusminus \tau(\chi) L\pb{\half, g \otimes \chibar} \chi(p) + (1-\eta_\pm)  c_g(q) L\pb{\half, g}}.
    \end{equation}
\begin{mycoro}
\label{maincor2} 
    For odd primes $p,q$ with $p < q^{1-\varepsilon}$ and $M \in \mz_{\geq 1}$, we have
    \begin{equation}
          M_g^{\pm}(p,q) \mp  M_g^{\pm}(q,p)  = (1-\eta_\pm)  L\pb{\half, g} + \sum_{m=1}^{M} b^\pm_{m,g}\pb{\frac{2 \pi  q}{p}}^{-m}+ O_M\pb{\pb{\frac{ q}{p}}^{-M-1}},
     \end{equation}
  where $b^\pm_{m,g}$ are constants depending only on $m$ and $g$.
\end{mycoro}
 
 \subsection{Comparison with previous results}
 
 As a starting point, our paper builds upon the recent work of Bainbridge, Khan, and Tang \cite{bkt}, where an asymptotic series for an untwisted version of Theorem \ref{thm:mainthm2} was established (that is, for $M_g^\pm(1,1;q)$). In this paper we introduce new ideas to be able to include twists and to obtain the exact formula (rather than an asymptotic series) in Theorem \ref{thm:mainthm}
 
For holomorphic Hecke cusp forms, Corollary \ref{maincor} was already known by the works of Bettin and Drappeau \cite[Lemma 9.3]{bettdrap} and Nordentoft \cite[Theorem 4.4 and Section 5.1]{nord}. However their results are only for a simple twist. Thus our Theorem \ref{thm:mainthm} extends their reciprocity relation to allow for general twisting. For Hecke Maass cusp forms, a reciprocity relation was previously established by Drappeau and Nordentoft in \cite[Corollary 1.7]{norddrap}, but again this only addresses the case of a simple twist, and moreover it is only valid up to a modest error term. Thus our Theorem \ref{thm:mainthm2} extends their result to general twists and allows for an arbitrarily sharp error term. 

Besides extending previous results, our paper also differs in methodology. The strategy of previous papers has been to show that the function
\[
\alpha \mapsto L\pb{\half, f,\alpha},
\]
for $\alpha \in\mathbb{Q}$, is a ``quantum modular form'' (see \cite[Definition 4.1]{nord}). While for simple twists this is proven elegantly in \cite{nord} in the holomorphic case, the proof in the Maass case as given in \cite{norddrap} is more difficult and requires additional geometric ideas. Moreover it is not clear to what extent the previous methods can handle general twists. Until now, only the work of Bettin \cite[Corollary 3]{bettin}, for the second moment of Dirichlet $L$-functions, had obtained reciprocity with general twisting, and this was only up to a modest error term. Bettin's ingenious method involved continued fraction expansions. Our new method on the other hand is relatively simple and results in an arbitrarily sharp error term. The only arithmetic inputs we will require are the Chinese Remainder Theorem, in the form of additive reciprocity (see Lemma \ref{lem:recipexp}), and the functional equations of the relevant $L$-functions. Another advantage of our new proof is that it works equally well in the holomorphic and Maass cases alike.

\begin{myremark} 
We have restricted to $GL(2)$ forms of level 1, but it would be interesting to generalize our reciprocity relations to congruence subgroups. Some results in this direction have been established in \cite{norddrap} and \cite{nord} by other methods. In order to apply our method, we would need to replace \eqref{eq:addtwistfe} by the corresponding functional equation in higher level. In general, the functional equation \cite[Theorem 3.1]{dhkl} is rather complicated and its practicability for our purposes is unclear. However, under some conditions (e.g. squarefree level coprime to $p, q, r$ with trivial nebentypus), the functional equation is much simpler (and it can also be derived by first passing to multiplicative twists as in Lemma \ref{lem:addtwistmoment}). Thus we expect our method to go through in at least some higher level settings.
%Under certain conditions (e.g. squarefree level coprime to $p, q, r$ with trivial nebentypus) one can derive a simple functional equation by first passing to multiplicative twists as in Lemma \ref{lem:addtwistmoment}, which suffice for our argument. However, when working in full generality, the functional equation \cite[Theorem 3.1]{dhkl} is rather complicated and its practicability for our purpose is unclear.
\end{myremark}

\subsection{Notation} 
\label{subsec:notations}
We use  $\sum^*$ to indicate that the sum is over primitive Dirichlet characters. We use $e(x)$ to denote $e^{2\pi i x}$. For a Dirichlet character $\chi$ modulo $q$, we write $\tau(\chi)$ for the Gauss sum $\sum_{n \mymod{q}} \chi(n)e(\frac{n}{q})$. We use $\varepsilon$ to denote an arbitrarily small positive constant, which may differ from one occurence to the next. When we write $e(\frac{\overline{a}}{b})$ for relatively prime integers $a$ and $b$, the notation $\overline{a}$ denotes an integer with $a\overline{a}\equiv 1 \mymod{b}$.

\section{Preliminaries}
\label{sec:background}
\subsection{Additive twists}
Let $f$ be a holomorphic Hecke cusp form of weight $k\ge 12$ for SL$_2(\mz)$. For $\alpha \in \mq$, the definition of $L(s, f, \alpha)$ given in \eqref{eq:addtwistdef} can be analytically continued to an entire function on $\mathbb{C}$, via the integral representation
\begin{equation}
\label{period}    L\pb{s-\frac{k-1}{2},f,\alpha} =\frac{(2\pi)^{s}}{\Gamma(s)} \int_{0}^{\infty} f(\alpha+iy)y^{s-1} dy.
\end{equation}
For relatively prime integers $a$ and $b$, we have the functional equation
\begin{equation}
\label{eq:addtwistfe}
    L\pb{\half+s, f,-\frac{\ov{a}}{b}} = i^{k} \pb{\frac{2\pi}{b}}^{2s} \frac{\Gamma(\frac{k}{2}-s)}{\Gamma(\frac{k}{2}+s)} L\pb{\half-s, f, \frac{a}{b}},
\end{equation}
We refer the reader to \cite[Section 3.3]{nord2} for further details. Additive twists are connected to the weighted moments that we study. To describe the connection, we define the entire functions
\begin{equation}
\label{eq:cfsdefn}
c_f(s,q) = \frac{\lambda_f(q)q^{\half}}{q^{s}} - \frac{1}{q^{2s}} - 1
\end{equation}
and
\begin{equation}
\label{eq:Mfspq}
    M_f(s,p,r;q)= \frac{1}{\phi(q)} \pb{\sumstarchiq \tau(\chi)L\pb{s+\half, f\otimes \chibar} \chi(p)\ov{\chi}(r) +c_f(s,q) L\pb{s+\frac12,f}}. \  
\end{equation}
Comparing with \eqref{eq:Mfdefn}, we immediately note that $c_f(0,q) =c_f(q)$, and $M_f(0,p,r;q) = M_f(p,r;q)$. We then have the following lemma.

\begin{mylemma}[Relating additive and multiplicative twists]
    \label{lem:addtwistmoment}
For all $s \in \mc$, we have
    \begin{equation}
    M_f(s,p,r;q) = L\pb{s+\half, f , \frac{\ov{p}r}{q}}.
    \end{equation}
\end{mylemma}

\begin{proof}
Using the orthogonality of the Dirichlet characters, we have 
       \begin{equation}
        \label{eq:gauss}
       \frac{1}{\phi(q)} \sumstarchiq \tau(\chi)\chi(m) =
        e\pb{\frac{\ov{m}}{q}} + \frac{1}{\phi(q)} 
            \end{equation}
            for $(m,q)=1$ (while the sum on the left hand side is 0 for $q|m$).
 Assume that $\Re(s) > \tfrac12$. We can use the absolute convergence of the Dirichlet series expansion and \eqref{eq:gauss} to get 
\begin{multline}
\label{eq:Mtfspq}
     \frac{1}{\phi(q)}  \sumstarchiq \tau(\chi)L\pb{ s+\frac12, f\otimes \chibar} \chi(p) \chibar(r) \\
     = \sumnq \frac{\lambda_f(n)}{n^{s+\half}} \frac{1}{\phi(q)} \sumstarchiq \tau(\chi) \chi(\ov{n}p\ov{r})  = \sumnq \frac{\lambda_f(n)}{n^{s+\half}} \pb{e{\pb{\frac{n\ov{p}r}{q}}} +\frac{1}{\phi(q)}}. 
\end{multline}
Next we want to extend the $n$-sum to all $n\geq 1$. Note that as $q$ is a prime, the Hecke eigenvalues satisfy the identity $\lambda_f(mq) = \lambda_f(m)\lambda_f(q) - \lambda_f(\frac{m}{q}) \delta(q \mid m)$. Using this, we have   
\begin{equation}
\label{eq:qdivn}
    \sum_{q \mid n} \frac{\lambda_f(n)}{n^{s+\half}} = \frac{\lambda_f(q)}{q^{s+\half}} \sum_{m \geq 1} \frac{\lambda_f(m)}{m^{s+\half}} - \frac{1}{q^{2s+1}}\sum_{q \mid m} \frac{\lambda_f({\frac{m}{q}})}{(\frac{m}{q})^{s+\half}} = \pb{\frac{\lambda_f(q)}{q^{s+\half}} - \frac{1}{q^{2s+1}} }L\pb{s+\half, f}.
\end{equation}
Combining \eqref{eq:Mtfspq} and \eqref{eq:qdivn}, and using that $e(\frac{n\ov{p}}{q})=1$ when $q \mid n$, we get
\begin{align}
\label{eq:Mtfspq2}
    \frac{1}{\phi(q)}&  \sumstarchiq \tau(\chi) L\pb{ s+\frac12, f\otimes \overline{\chi}} \chi(p)\chibar(r) \\ 
    \nonumber &=  \sumn \frac{\lambda_f(n)}{n^{s+\half}} e{\pb{\frac{n\ov{p}r}{q}}} +\frac{1}{\phi(q)}L\pb{ s+\frac12, f}- \pb{1+\frac{1}{\phi(q)}}\sum_{q \mid n} \frac{\lambda_f(n)}{n^{s+\half}} \\ 
    \nonumber &= \sumn \frac{\lambda_f(n)}{n^{s+\half}} e{\pb{\frac{n\ov{p}r}{q}}} - \frac{1}{\phi(q)} \pb{\frac{\lambda_f(q)q^{\half}}{q^{s}} - \frac{1}{q^{2s}} - 1}L\pb{ s+\frac12, f}.
\end{align}
Using the definition of $c_f(s,q)$ in \eqref{eq:cfsdefn}, we can then rewrite \eqref{eq:Mtfspq2} as
\begin{equation}
\label{eq:Mtfspq3}
    M_f(s,p,r;q) = \sumn \frac{\lambda_f(n)}{n^{s+\half}}e{\pb{\frac{n\ov{p}r}{q}}} = L\pb{s+\half, f ,\frac{\ov{p}r}{q}}.
\end{equation}
Since both $ M_f(s,p,r;q)$ and $  L(s, f, \frac{\ov{p}r}{q})$ are entire functions, \eqref{eq:Mtfspq3} must be true for all $s \in \mc$, by the principle of analytic continuation.  
\end{proof}
We will need the following standard bound.
\begin{mylemma}[Bounds on $L$-functions in the $t$-aspect]
    \label{lem:convexity}
For $s=\sigma+it \in \mc$, we have the $t$-aspect bound
    \begin{equation}
L\pb{s, f , \frac{\ov{p}r}{q}} \ll (1+|t|)^\theta,
    \end{equation}
    where $\theta=0$ for $\sigma>1$, $\theta=1-\sigma$ for $0\le \sigma \le 1$, and $\theta = 1 -2\sigma$ for $\sigma<0$. The implied constant in the bound may depend on $f, p , r, q, \sigma$. 
\end{mylemma}
\begin{proof}
By definition \eqref{eq:Mfspq} and Lemma \ref{lem:addtwistmoment}, it suffices to appeal to $t$-aspect bounds for $L(s, f\otimes \chi)$. For $0\le \sigma \le 1$, we use the convexity bound for this degree 2 $L$-function. For $\sigma<0$ we use the functional equation to relate the $L$-function to values to the right of the critical strip, where it is easily bounded.
\end{proof}

We will need the following ``additive reciprocity'' formula for the exponential function, which follows from Chinese Remainder Theorem.
\begin{mylemma}[Additive reciprocity]
    \label{lem:recipexp}
    Let $n,a,b$ be integers, with $(a,b)$=1. Then, we have
    \begin{equation}
         e\pb{\frac{n\ov{a}}{b}}= e\pb{-\frac{n\ov{b}}{a}}    e\bigg(\frac{n}{ab}\bigg). 
    \end{equation}
\end{mylemma}

\subsection{Analytic tools}

We will need Stirling's approximation for the gamma function (see \cite[Equation (3.13)]{bkt}).
\begin{mylemma}[Stirling's approximation]\label{lem:stirl} Fix a compact set $\mathcal{C}\subset \{z\in \mathbb{C}: \Re(z)\ge 0\}$. For $z\in\mathcal{C}$ and $t\in\mathbb{R}$ with $|t|>\half$, we have
\begin{multline*}
\Gamma(z+it)\\
=\sqrt{2\pi} e^{-i\frac{\pi}{4}\sgn(t)} |t|^{z-\half} \exp\pb{ -\frac{\pi}{2} |t| + it\log|t| -it+iz\frac{\pi}{2}\sgn(t)} \pb{ 1+ \sum_{m=1}^{M} a_m(z)t^{-m} + E_M(z,t)},
\end{multline*}
where $a_m(z)$ and $E_M(z,t)$ are holomorphic functions of $z\in\mathcal{C}$
which depend on $\sgn(t)$ and satisfy the bounds $|a_m(z)|_M \ll 1$ and $|E_M(z,t)|\ll_M |t|^{-M-1}.$
\end{mylemma}
\noindent We will often use the upper bound $|\Gamma(z+it)|\ll_{z} (1+|t|)^{\Re(z)-\half} \exp( -\frac{\pi}{2} |t|)$ for $\Re(z)\not\in\mathbb{Z}_{\le 0}$.

We will also need a stationary phase expansion of oscillatory integrals (see \cite[Proposition 8.2]{bky}). 
\begin{mylemma}[Stationary phase]\label{lem:statph} 
Fix $W(y)$ a smooth function of $y\in \mathbb{R}$ with compact support. For a parameter $x>0$, let $\phi_x(y)$ be a function which is smooth on $\mathrm{supp}(W)$, where exists a unique point $y_0$ with $\phi_x'(y_0)=0$, and which satisfies $\phi_x^{\prime\prime}(y)\gg x$ and $\phi_x^{(j)}(y)\ll x$ for all $j\in \mathbb{Z}_{\ge 1}$. Then
\[
\int_{-\infty}^{\infty} W(y) e^{i\phi_x(y)} dy = \pb{\frac{2\pi}{\phi_x^{\prime\prime}(y_0)}}^\half e^{i\phi_x(y_0)+i\frac{\pi}{4}} \sum_{m=0}^{100M} \frac{1}{m!} \pb{\frac{i}{2\phi_x^{\prime\prime}(y_0)}}^m V_x^{(2m)}(y_0) + E_M(x),
\]
where
\[
V_x(y)=W(y)\exp\pb{ i \pb{ \phi_x(y)-\phi_x(y_0)-\half \phi_x^{\prime \prime}(y_0)(y-y_0)^2}}
\]
and $E_M(x)\ll_M x^{-M}$.
\end{mylemma}

\subsection{Integral transforms and calculations}

We list some important lemmas associated to some integral transforms we will encounter later. 

\begin{mylemma}[Mellin transform] 
\label{lem:mellinrep} 
    Let $d \in (0,1)$. For $x > 0$ we have
    \begin{equation}
    \label{eq:expmellin}
    e\left(x\right) = \frac{1}{2\pi i}\int_{(-d)}  \Gamma(w) e^{i \frac{\pi}{2} w } (2\pi x)^{-w} dw + 1.
    \end{equation}
    This integral is absolutely convergent if $d \in (\half,1)$.
\end{mylemma}
\begin{proof}
    The proof is just like \cite[Lemma 7]{bkt}. 
    \end{proof}

\begin{mylemma}[A residue calculation] 
\label{lem:residues}
    Let $x>0, \alpha \in \mq$ and $s, w \in \mc$. For $n \in \mathbb{Z}_{\geq 0}$, we have
    \begin{equation}
\underset{w=-n}{\mathrm{Res}}  \  \Gamma(w) e^{i \frac{\pi}{2} w }  x^w L\pb{\half+s+w, f, \alpha} = \frac{i^n}{n!} x^{-n}  L\pb{\half+s-n, f , \alpha}.
  \end{equation}
Furthermore, the residue is $0$ if and only if there is no pole. 
\end{mylemma}
\begin{proof}
   This follows from the fact that  $\Gamma(w)$ has a simple pole at $w=-n$ for $n \in \mathbb{Z}_{\geq 0}$, with residue given by 
 \begin{align}
 \label{eq:gammares}
 \underset{w=-n}{\mathrm{Res}} \ \Gamma(w) = \frac{(-1)^n}{n!},
 \end{align}
    while $e^{i \frac{\pi}{2} w }  x^w L(\half + s+w, f, \alpha)$ is entire in $w$.  Thus $\Gamma(w) e^{i \frac{\pi}{2} w }  x^w L\pb{\half+s+w, f, \alpha}$ either has a simple pole at $w=-n$ or is holomorphic there, depending on whether $L(\half+s-n, f, \alpha)$ vanishes.
        \end{proof}

Define the sets
\begin{align}
\label{eq:setsdef} &\mcs_1 = \left\{ s \in \mathbb{C} : -1 < \Im(s) < 1,\; \frac{19}{16} < \Re(s) < \frac{11}{8} \right\},\\
\nonumber &\mcs_2 = \left\{ s \in \mathbb{C} : -1 < \Im(s) < 1,\; -\frac{1}{10} < \Re(s) <  \frac{11}{8} \right\}.
\end{align}

\begin{mylemma}[An approximate evaluation of an integral]
\label{lem:statphase}
Let $k\ge 12$ be a fixed even integer. For $s\in \mcs_1$ and $x>0$, we have
\begin{align}
\label{eq:int-def}  I(x,s):= \frac{1}{2\pi i} \int_{(-\frac{15}{8})} \Gamma(w) e^{i \frac{\pi}{2} w} x^{w} \frac{\Gamma(\frac{k}{2}-s-w)}{\Gamma(\frac{k}{2}+s+w)}dw = i^k x^{-2s}e^{-ix}  + Q(x,s),
\end{align}
for some function $Q(x,s)$ which is holomorphic for $s$ in the larger set $\mcs_2$ and satisfies the bound 
\begin{equation}
\label{eq:Qdecay} Q(x,s) \ll x^{-2\Re(s)-1} + x^{-\frac{15}{8}}.
\end{equation}
\end{mylemma}
\begin{proof} Such a result is implicit in \cite[Section 5]{bkt}, but for completeness we sketch the argument here. Write $w=-\frac{15}{8}+it$. Restricted to $|t|\le 1$, the finite integral is bounded absolutely by $O(x^{-\frac{15}{8}} )$. For $|t|>1$, by Stirling's formula  for the gamma function (Lemma \ref{lem:stirl}), we have 
\begin{align}
\label{eq:stirling-expansion}\Gamma(w) e^{i \frac{\pi}{2} w }  x^{w}  \frac{\Gamma(\frac{k}{2}-s-w)}{\Gamma(\frac{k}{2}+s+w)}=&  
e^{-\tfrac{\pi}{2}(t+|t|)} \sqrt{2\pi} x^{-\frac{15}{8}} |t|^{\frac{11}{8}-2s} i^k  \exp\pb{ - i t\log \pb{\frac{|t|}{xe}}-i \frac{\pi}{4} }\\
\nonumber &+(\text{lower order terms}) + (\text{remainder}).
\end{align}
The lower order terms, decreasing by factors of $|t|^{-1}$, are of a similar shape as the leading term and are analytic for $s\in \mathcal{S}_2$. If we take enough terms in the Stirling expansion then the remainder, which is necessarily analytic in $ \mathcal{S}_2$ since all other terms in \eqref{eq:stirling-expansion} are, is of size $O(x^{-\frac{15}{8}} |t|^{-10})$ and thus contributes $O(x^{-\frac{15}{8}} )$ to $I(x,s)$. We consider the contribution of the leading term:
\[
\frac{1}{2\pi } \int_{|t|>1} e^{-\tfrac{\pi}{2}(t+|t|)} \sqrt{2\pi} x^{-\frac{15}{8}} |t|^{\frac{11}{8}-2s} i^k  \exp\pb{ - i t\log \pb{\frac{|t|}{xe}}-i \frac{\pi}{4} }  dt .
\]
First we note that the integral converges absolutely since $\frac{11}{8}-2\Re(s)<\frac{11}{8}-2(\frac{19}{16})=-1$ by assumption. We also see that the integrand decays exponentially as $t\to +\infty$. So for $t> 1$, the integral is defined and holomorphic for $s\in \mathcal{S}_2$, and we may bound it absolutely by $O(x^{-\frac{15}{8}} )$. We now restrict the integral to $t<-1$, where $|t|=-t$ and $e^{-\tfrac{\pi}{2}(t+|t|)} =1$, and consider two cases.

First, suppose that $x>2$. By a partition of unity, we may study the integral in dyadic intervals:
\begin{align*}
 I_{j}(x,s):= \frac{1}{2\pi } \int_{-\infty}^{\infty} \sqrt{2\pi} x^{-\frac{15}{8}} (-t)^{\frac{11}{8}-2s} i^k  \exp\pb{ - i t\log \pb{\frac{-t}{xe}}-i \frac{\pi}{4} }  W\pb{\frac{t}{-T_j}} dt,
 \end{align*}
 where $W(y)$ is a smooth bump function with compact support on $\mathbb{R}_{>0}$ such that $W(1)=1$, and $T_j=x2^j$ for integers $j\ge -\lceil \log_2 x \rceil $.  By the substitution $y= -\frac{t}{T_j}$ we have
\begin{align*}
 I_{j}(x,s)=(2\pi)^{-\half} x^{-\frac{1}{2}-2s} i^k e^{-i \frac{\pi}{4}} T_j  \int_{-\infty}^{\infty}  \pb{\frac{T_j y}{x}}^{\frac{11}{8}-2s}   \exp\pb{  i y T_j \log \pb{\frac{yT_j}{xe}}}  W(y) dy, 
\end{align*}
For $j\neq 0$, if $\mathrm{supp}(W)$ is chosen appropriately, we have that $|\frac{yT_j}{x}-1|\gg 1$, so $|\frac{d}{dy} (y T_j \log (\frac{yT_j}{xe}))| = |T_j \log (\frac{yT}{x})|\gg T_j$ and $|\frac{d^n}{dy^n} (y T_j \log (\frac{yT_j}{xe}))|\asymp T_j$ for $n\ge 2$. By integration by parts multiple times, we have the bound
\[
\sum_{ j\neq 0} |I_j(x,s)| \ll \sum_{ j\neq 0}  (T_j)^{-10} x^{-\frac{15}{8}} \ll x^{-\frac{15}{8}}.
\]
As for analyticity, since each finite integral $I_j(x,s)$ is holomorphic in $\mathcal{S}_2$ and we have shown absolute convergence of the sum of these integrals for $j\neq 0$, we can deduce that the sum itself is holomorphic in $\mathcal{S}_2$. It remains to consider the special case $j=0$ (for which $T_0 = x$). We have
\begin{align}
\label{I0} I_0(x,s)=(2\pi)^{-\half} x^{\frac{1}{2}-2s} i^k e^{-i \frac{\pi}{4}}  \int_{-\infty}^{\infty} y^{\frac{11}{8}-2s}   \exp\pb{ i y x \log \pb{\frac{y}{e}}}  W(y) dy.
\end{align}
The phase $\phi(y)= yx\log(\frac{y}{e})$ now has a stationary point at $y=1$, with $\phi(1)= -x$ and $\phi''(1)=x$. Evaluating the integral by a stationary phase expansion using Lemma \ref{lem:statph} yields 
\begin{align}
\label{eq:stat-expansion} \int_{-\infty}^{\infty} y^{\frac{11}{8}-2s}   \exp\pb{ i y x \log \pb{\frac{y}{e}}}  W(y) dy = e^{i \frac{\pi}{4} } \sqrt{\frac{2\pi}{x}} \exp(-ix) +\text{(lower order terms)} + \text{(remainder)}.
\end{align}
The lower order terms, decreasing by factors of $x^{-1}$, are bounded by $x^{-\frac32}$, of a similar shape as the leading term, and are analytic for $s\in \mathcal{S}_2$. If we take enough terms in the stationary phase expansion then the remainder, which is necessarily analytic in $ \mathcal{S}_2$ since all other terms in \eqref{eq:stat-expansion} are, is of size $O(x^{-10})$. Inserting \eqref{eq:stat-expansion} into \eqref{I0}, the contribution of the leading term is $i^k x^{-2s}e^{-ix}$ and the contribution of the rest is $O(x^{-2\Re(s)-1})$.

Now consider the second case, when $0<x\le 2$. We can then bound the integral with $t\ge -3$ absolutely, and for $t<-3$ we can use a partition of unity and argue as before. However since $-x\not\in (-\infty, -3)$, this time there is no stationary point, and so we obtain $I(x,s)\ll x^{-\frac{15}{8}}\ll 1$. It is still a true statement to write $I(x,s)=i^k x^{-2s}e^{-ix}+O(x^{-2\Re(s)-1}+x^{-\frac{15}{8}})$ since $ x^{-2\Re(s)}\ll x^{-2\Re(s)-1}$ for $0<x\le 2$.

\end{proof}

\begin{mylemma}[An exact evaluation of an integral]
\label{lem:keyintegral}   Let $k\ge 12$ be a fixed even positive integer. For $x>0$, we have 
    \begin{equation}
    \label{eq:keyint}
    J(x) :=   \frac{1}{2\pi i} \int_{(\frac34)} \Gamma(w) \, e^{\frac{i\pi w}{2}} \, x^{w} 
        \frac{\Gamma \left(  \frac{k}{2} - w \right)}{\Gamma \left(  \frac{k}{2} + w \right)} \, dw 
        = \sum_{j=0}^{ \frac{k}{2}-1}  (ix)^{-j} \frac{( \frac{k}{2}-1+j)!}{j!( \frac{k}{2}-1-j)!} \pb{(-1)^j + i^k e^{-ix}}.
        \end{equation}
    \end{mylemma}
    \begin{proof}
      Writing $w=\frac34+it$, Stirling's estimates (see \eqref{eq:stir} with $s=0$ and $\Re(w)=\frac34$) imply that the integrand is bounded by $x^\frac34 |t|^{-\frac54}$. Thus the integral $J(x)$ converges absolutely.
        Using the formulas $\Gamma(z+1)=z\Gamma(z)$ and $\sin(\pi z) \Gamma(z)\Gamma(1-z)=\pi$, we have
       \begin{align}
            \label{eq:gammamod} \Gamma(w)\frac{\Gamma \left( \frac{k}{2} - w \right)}{\Gamma \left( \frac{k}{2} + w \right)}&=\Gamma(w)\frac{\Gamma \left( 1 - w \right)}{\Gamma\left( \frac{k}{2} + w \right)}\prod_{j=1}^{\frac{k}{2}-1}(j-w)=\frac{\pi}{ (-1)^{\frac{k}{2}}\sin(\pi (\frac{k}{2}+w)) \Gamma \left( \frac{k}{2} + w \right)}\prod_{j=1}^{\frac{k}{2}-1}(j-w)\\
           \nonumber &=\Gamma\left(1-w-\tfrac{k}{2}\right) (-1)^{\frac{k}{2}}\prod_{j=1}^{\frac{k}{2}-1}(j-w)=-\Gamma\left(1-w-\tfrac{k}{2}\right)\prod_{j=1}^{\frac{k}{2}-1}(w-j).
       \end{align}
Now using \eqref{eq:gammamod}, we can express the left side of \eqref{eq:keyint} as
        \begin{align}
        \label{eq:Ixver1}
         J(x) &= -\frac{1}{2\pi i}\int_{(\frac34)} \Gamma(1-w-\tfrac{k}{2}) \, e^{\frac{i\pi w}{2}} \, x^{w} \prod_{j=1}^{\frac{k}{2}-1}(w-j) \ dw \\
     \nonumber       &=-x^{\frac{k}{2}}\left(\frac{d}{dx}\right)^{\frac{k}{2}-1}\pb{\frac{1}{2\pi i}\int_{(\frac34)} \Gamma(1-w-\tfrac{k}{2}) \, e^{\frac{i\pi w}{2}} \, x^{w-1} dw},
        \end{align}
        where absolute convergence allows us to differentiate under the integral sign. Making the substitution $u=1-w-\frac{k}{2}$, we rewrite \eqref{eq:Ixver1} as
        \begin{equation}
        \label{eq:Ixver2}
           J(x) = -x^{\frac{k}{2}}e^{-\frac{ i\pi}{2}(\frac{k}{2}-1)}\left(\frac{d}{dx}\right)^{\frac{k}{2}-1}\pb{\frac{x^{-\frac{k}{2}}}{2\pi i}\int_{(\frac14-\frac{k}{2})} \Gamma(u) \, e^{-i\frac{\pi}{2}u } \, x^{-u} du}.
        \end{equation}
       We now shift the line of integration to the right, from $\Re(u)=\frac14-\tfrac{k}{2}$ to $\Re(u)=-\frac34$, using Stirling's estimate to justify absolute convergence, and picking up residues (actually negative of the residues since we are shifting right) from the poles of $\Gamma(u)$ at $u=-j$ for $j=1,2,\cdots,\frac{k}{2}-1$. This gives
     \begin{align*}
         J(x) =-x^{\frac{k}{2}}e^{-\frac{ i\pi}{2}(\frac{k}{2}-1)}\left(\frac{d}{dx}\right)^{\frac{k}{2}-1}\pb{ -\sum_{j=1}^{\frac{k}{2}-1} \frac{(-1)^j}{j!} i^j x^{j-\frac{k}{2}} + \frac{x^{-\frac{k}{2}}}{2\pi i}\int_{(-\frac34)} \Gamma(u) \, e^{-i\frac{\pi}{2}u } \, x^{-u} du},
         \end{align*}
         which on using Lemma \ref{lem:mellinrep} yields
      \begin{align}
        \label{eq:Ixver3}
         J(x) &= -x^{\frac{k}{2}}e^{-\frac{ i\pi}{2}(\frac{k}{2}-1)}\left(\frac{d}{dx}\right)^{\frac{k}{2}-1}\pb{ -\sum_{j=1}^{\frac{k}{2}-1} \frac{(-1)^j}{j!} i^j x^{j-\frac{k}{2}} + x^{-\frac{k}{2}}(e^{-ix}-1)} \\ 
         \nonumber   &= -x^{\frac{k}{2}}e^{-\frac{ i\pi}{2}(\frac{k}{2}-1)}\left(\frac{d}{dx}\right)^{\frac{k}{2}-1}\pb{ -\sum_{j=0}^{\frac{k}{2}-1} \frac{(-1)^j}{j!} i^j x^{j-\frac{k}{2}} + x^{-\frac{k}{2}}e^{-ix}}.
              \end{align}
Using the Leibniz rule for differentiation, we obtain
    \begin{multline}
        \label{eq:Ixver4}
        J(x) = x^{\frac{k}{2}}i^{-(\frac{k}{2}-1)} \left( \sum_{j=0}^{\frac{k}{2}-1}\frac{(-i)^j}{j!} x^{j-k+1} \pb{j-\frac{k}{2}}\pb{j-\frac{k}{2}-1}\cdots\bigg(j-k+2\bigg)   \right. \\ \left. - \sum_{j=0}^{\frac{k}{2}-1}  \binom{\frac{k}{2}-1}{j}  \pb{\frac{k}{2}}\pb{\frac{k}{2}+1}\cdots \pb{\frac{k}{2}+j-1} (-1)^j x^{-\frac{k}{2}-j} (-i)^{\frac{k}{2}-1-j}e^{-ix} \vphantom{\frac12} \right).
    \end{multline}
Replacing $j$ with $\frac{k}{2}-j-1$ in the first sum, and then simplifying, we obtain the result.   
  \end{proof}

\section{Proof of Theorem \ref{thm:mainthm}}

Recall definition \eqref{eq:setsdef} and assume that $s \in \mcs_1$.  Using Lemmas \ref{lem:addtwistmoment}, \ref{lem:recipexp}, and \ref{lem:mellinrep}, we have

\begin{align}
\label{eq:Mfsprq2}
    M_f(s,p,r;q) &= \sum_{n \geq 1} \frac{\lambda_f(n)}{n^{\half+s}}e\pb{\frac{n\ov{p}r}{q}} =\sum_{n \geq 1} \frac{\lambda_f(n)}{n^{\half+s}}e\pb{-\frac{n\ov{q}r}{p}}e\pb{\frac{nr}{pq}} \\
    \nonumber &= \sum_{n \geq 1} \frac{\lambda_f(n)}{n^{\half+s}}e\pb{-\frac{n\ov{q}r}{p}} \pb{1+\frac{1}{2\pi i}\int_{(-\frac58)} \Gamma(w) e^{i \frac{\pi}{2} w }\pb{\frac{2\pi n r}{pq}}^{-w} dw} .
\end{align}
By the absolute convergence of the sum (since $\Re(s)>\frac{19}{16}>\half$) and integral (by Stirling's estimates), we may exchange the order of the summation and the integration, and use Lemma \ref{lem:addtwistmoment} once again to get
\begin{equation}
\label{eq:Mfsprq3}
      M_f(s,p,r;q)- M_f(s,-q,r;p)=  \frac{1}{2\pi i}\int_{(-\frac58)}   \Gamma(w) e^{i \frac{\pi}{2} w } \pb{\frac{2\pi r}{pq}}^{-w} L\pb{s+w+\half, f , -\frac{\ov{q}r}{p}}  dw, 
\end{equation}
where we used that $\Re(s+w+\half)>\frac{19}{16}-\frac59+\half>1$ in the integrand to form the $L$-function.

Next we'd like to shift the integral in \eqref{eq:Mfsprq3} to the left to $\Re(w)=-\frac{15}{8}$, but first we check that the integral on any vertical line with $-\frac{15}{8}\le \Re(w)\le -\frac58$ will converge absolutely. For this, we write $t=\Im(w)$ and use Stirling's estimates and Lemma \ref{lem:convexity}. For $\Re(s+w+\half)>0$ and $|t|>1$, we have the $t$-aspect bound 
\[
\left\lvert \Gamma(w) e^{i \frac{\pi}{2} w }  L\pb{s+w+\half, f , -\frac{\ov{q}r}{p}} \right\rvert \ll  |t|^{\Re(w)-\half} |t|^{\max\{1-\Re(s+w+\half),0\}}= |t|^{\max \{-\frac{19}{16}, -\frac98\} }.
\]
For $\Re(s+w+\half)\le 0$, we have the $t$-aspect bound
\[
\left\lvert  \Gamma(w) e^{i \frac{\pi}{2} w }  L\pb{s+w+\half, f, -\frac{\ov{q}r}{p}} \right\rvert  \ll  |t|^{\Re(w)-\half} |t|^{1-2\Re(s+w+\half)} < |t|^{\frac{15}{8}-\half -2\Re(s)}.
\]
The final exponent here is strictly less than $-1$, since $\Re(s)>\frac{19}{16}$.

Having established absolute convergence, we now move the integral in \eqref{eq:Mfsprq3} to $\Re(w)=-\frac{15}{8}$, possibly crossing a pole at $w=-1$. By Lemma \ref{lem:residues}, we 
get
\begin{equation}
\label{eq:Mfsprq4}
      M_f(s,p,r;q)-M_f(s,-q,r;p)=  \frac{2 \pi i r}{pq} L\pb{ s-\half, f , -\frac{\ov{q}r}{p}} + H(s),
\end{equation}
where
\begin{equation}
    \label{eq:Js}
   H(s) = \frac{1}{2 \pi i} \int_{(-\frac{15}{8})} \Gamma(w) e^{i \frac{\pi}{2} w } \pb{\frac{2\pi r}{pq}}^{-w}  L\pb{s+w+\half, f , -\frac{\ov{q}r}{p}} dw.
\end{equation}
It is clear that every term in \eqref{eq:Mfsprq4} except possibly $H(s)$ is holomorphic in the wider set $s\in \mcs_2$. Our eventual goal is to obtain an analytic continuation of $H(s)$ to $s\in \mcs_2$. Then we will be able to set $s=0$, and obtain the desired identity in Theorem \ref{thm:mainthm}.

Using the functional equation \eqref{eq:addtwistfe}, we get
\begin{equation}
    \label{eq:Js2}
   H(s)=  \frac{1}{2 \pi i} \int_{(-\frac{15}{8})} \Gamma(w) e^{i \frac{\pi}{2} w } \pb{\frac{2\pi r}{pq}}^{-w} i^k \pb{\frac{2\pi}{p}}^{2s+2w} \frac{\Gamma(\frac{k}{2}-s-w)}{\Gamma(\frac{k}{2}+s+w)} 
   L\pb{\half-s-w, f , \frac{q \ov{r}}{p}} dw.
\end{equation}
We have that $\Re(\thalf-s-w)>\thalf -\frac{11}{8}-(-\frac{15}{8})>1$, so we can write the additively twisted $L$-series as an absolutely convergent sum. Then changing the order of the sum and integral, we arrive at 

\begin{equation}
    \label{eq:Js3} 
   H(s) =  i^k  \pb{\frac{2\pi}{p}}^{2s}  \sum_{n \geq 1} \frac{\lambda_f(n)}{n^{\half-s}} e\pb{\frac{nq\ov{r}}{p}} I\pb{\frac{2\pi n q}{pr},s},
\end{equation}
where $I(x,s)$ was defined in \eqref{eq:int-def}. We may write
\begin{align}
    \label{eq:I-rewrite} 
I\pb{\frac{2\pi n q}{pr},s} &= I\pb{\frac{2\pi n q}{pr},s} -  i^k n^{-2s} \pb{\frac{2\pi q}{pr}}^{-2s}e\pb{-\frac{nq}{pr}} +  i^k n^{-2s} \pb{\frac{2\pi q}{pr}}^{-2s}e\pb{-\frac{nq}{pr}}. 
%\\ \nonumber &=Q\pb{\frac{2\pi n q}{pr},s} +  i^k n^{-2s} \pb{\frac{2\pi q}{pr}}^{-2s}e\pb{-\frac{nq}{pr}},
\end{align}
%where $Q(x,s)$ is the function from Lemma \ref{lem:statphase}. 
Inserting \eqref{eq:I-rewrite} into \eqref{eq:Js3}  and using Lemma \ref{lem:recipexp}, we get
\begin{align}
\label{eq:J-rewrite}
    H(s) =  i^k \pb{\frac{2\pi}{p}}^{2s}  \sum_{n \geq 1} \frac{\lambda_f(n)}{n^{\half-s}} e\pb{\frac{nq\ov{r}}{p}}Q\pb{\frac{2\pi n q}{pr},s} + \pb{\frac{q}{r}}^{-2s} L\pb{s+\half, f,\frac{-q\ov{p}}{r}}.
\end{align}
Plugging \eqref{eq:J-rewrite} back into \eqref{eq:Mfsprq4} and using Lemma \ref{lem:addtwistmoment}, we obtain that for $s\in \mathcal{S}_{1}$ we have
\begin{multline}
\label{eq:clean} 
M_f(s,p,r;q)-M_f(s,-q,r;p) -  \pb{\frac{q}{r}}^{-2s} M_f(s,-p,q;r) \\
=  \frac{2 \pi i r}{pq} L\pb{ s-\half, f , -\frac{\ov{q}r}{p}}  + i^k \pb{\frac{2\pi}{p}}^{2s}  \sum_{n \geq 1} \frac{\lambda_f(n)}{n^{\half-s}} e\pb{\frac{nq\ov{r}}{p}}Q\pb{\frac{2\pi n q}{pr},s}. 
\end{multline}
By Lemma \ref{lem:statphase}, we have that $Q(\frac{2\pi n q}{pr},s)$ is analytic for $s$ in the wider set $\mathcal{S}_2$, and by its decay property \eqref{eq:Qdecay} we have that the sum $\sum_{n\ge 1}$ on the right hand side of \eqref{eq:clean} converges absolutely for $s\in \mathcal{S}_2$, so that the sum itself is analytic for $s\in \mathcal{S}_2$. Thus, by the principle of analytic continuation, the identity \eqref{eq:clean} holds for $s\in \mathcal{S}_2$. In particular, at $s=0$, we have
\begin{multline} \label{eq:cleaner}
 M_f(p,r;q)-M_f(-q,r;p) -   M_f(-p,q;r) \\
=  \frac{2 \pi i r}{pq} L\pb{-\half, f , \frac{-\ov{q}r}{p}} + i^k   \sum_{n \geq 1} \frac{\lambda_f(n)}{n^{\half}} e\pb{\frac{nq\ov{r}}{p}}Q\pb{\frac{2\pi n q}{pr},0}. 
\end{multline}

We now evaluate $Q(\frac{2\pi n q}{pr},0)$. 
\begin{mylemma} \label{lem:Qidentity} Let $k\ge 12$ be a fixed even integer. We have
\begin{align*}
Q\pb{\frac{2\pi n q}{pr},0} = -  i\frac{pr}{2 \pi q n }\frac{\Gamma(\frac{k}{2}+1)}{\Gamma(\frac{k}{2}-1)}+ \sum_{j=1}^{\frac{k}{2}-1}  \pb{\frac{2\pi i q n }{pr}}^{-j}  \frac{(\frac{k}{2}-1+j)!}{j!(\frac{k}{2}-1-j)!} \pb{(-1)^j + i^k e\pb{-\frac{qn}{pr}} }. 
\end{align*}
\end{mylemma}
\noindent We postpone the proof of this result and first derive the consequence that we need.
Applying Lemma \ref{lem:Qidentity} to subsitute for $Q\pb{\frac{2\pi n q}{pr},0}$ in \eqref{eq:cleaner}, and then using Lemma \ref{lem:recipexp}, we get
\begin{align*}
&M_f(p,r;q)-M_f(-q,r;p) -   M_f(-p,q;r)  \\
&=  \frac{2 \pi i r}{pq} L\pb{-\half, f , \frac{-\ov{q}r}{p}} 
-  i^{k+1}  \frac{pr}{2 \pi q}\frac{\Gamma(\frac{k}{2}+1)}{\Gamma(\frac{k}{2}-1)} L\pb{\frac32, f, \frac{q\ov{r}}{p}} \\
&+\sum_{j=1}^{\frac{k}{2}-1} \frac{1}{j!} \frac{\Gamma(\frac{k}{2}+j)}{\Gamma(\frac{k}{2}-j)} \pb{\frac{2\pi i q}{pr}}^{-j} \pb{ L\pb{ \half +j, f , \frac{-\ov{p}q}{r}} + i^k(-1)^j L\pb{ \half +j, f , \frac{q\ov{r}}{p}}}.
\end{align*}
After observing that the first two terms on the right hand side cancel each other out by the functional equation \eqref{eq:addtwistfe}, we obtain the desired identity \eqref{eq:thmhol} and complete the proof of Theorem \ref{thm:mainthm}.

\subsection{Proof of Lemma \ref{lem:Qidentity}}

Recall the definition
\begin{align*}
Q\pb{\frac{2\pi n q}{pr},s} = \frac{1}{2\pi i} \int_{(-\frac{15}{8})} \Gamma(w) e^{i \frac{\pi}{2} w} \pb{\frac{2\pi n q}{pr}}^{w} \frac{\Gamma(\frac{k}{2}-s-w)}{\Gamma(\frac{k}{2}+s+w)}dw - i^k  \pb{\frac{2\pi n q}{pr}}^{-2s}  e\pb{-\frac{nq}{pr}},
\end{align*}
from Lemma \ref{lem:statphase}, valid for $s\in\mathcal{S}_1$. The integral converges absolutely on any fixed vertical line with $\Re(w)\ge -\frac{15}{8}$ because by Stirling's estimates we have
\begin{align}
\label{eq:stir} \left|\Gamma(w) e^{i \frac{\pi}{2} w} \frac{\Gamma(\frac{k}{2}-s-w)}{\Gamma(\frac{k}{2}+s+w)} \right| \ll (1+|t|)^{-\Re(w)-\half-2\Re(s)},
\end{align}
where $t=\Im(w)$. For $s\in\mathcal{S}_1$ and $\Re(w)\ge -\frac{15}{8}$, the exponent above is strictly less than $-1$. We shift the line of integration to the right, to $\Re(w) = \frac34$, in the process crossing poles of $\Gamma(w)$ at $w=-1$ and $w=0$ (we do not encounter any poles of $\Gamma(\frac{k}{2}-s-w)$ as $k\ge 12$). Keeping in mind that we pick up negative of the residues since we are shifting contours to the right, we get by \eqref{eq:gammares} that
\begin{align}
\label{eq:equality} Q\pb{\frac{2\pi n q}{pr},s} &= \frac{1}{2\pi i} \int_{(\frac34)} \Gamma(w) e^{i \frac{\pi}{2} w} \pb{\frac{2\pi n q}{pr}}^{w} \frac{\Gamma(\frac{k}{2}-s-w)}{\Gamma(\frac{k}{2}+s+w)}dw \\
\nonumber &- i\frac{pr}{2 \pi n q}\frac{\Gamma(\frac{k}{2}-s+1)}{\Gamma(\frac{k}{2}+s-1)} - \frac{\Gamma(\frac{k}{2}-s)}{\Gamma(\frac{k}{2}+s)} - i^k  \pb{\frac{2\pi n q}{pr}}^{-2s} e\pb{-\frac{qn}{pr}} 
\end{align}
for $s\in \mathcal{S}_1$. The right hand side now is defined and analytic for $s$ in the wider set $\mathcal{S}_2$, because by \eqref{eq:stir}, the shifted integral on the line $\Re(w)=\frac34$ converges absolutely for $s\in \mathcal{S}_2$. Since $Q(\frac{2\pi n q}{pr},s)$ is also analytic for $s\in \mathcal{S}_2$ by Lemma \ref{lem:statphase}, we have by the principle of analytic continuation that the equality \eqref{eq:equality} holds for $s\in \mathcal{S}_2$. In particular, at $s=0$,
\begin{align}
\label{eq:Q0identity} Q\pb{\frac{2\pi n q}{pr},0} &= \frac{1}{2\pi i} \int_{(\frac34)} \Gamma(w) e^{i \frac{\pi}{2} w} \pb{\frac{2\pi n q}{pr}}^{w} \frac{\Gamma(\frac{k}{2}-w)}{\Gamma(\frac{k}{2}+w)}dw 
+\frac{pr}{2 \pi i n q}\frac{\Gamma(\frac{k}{2}+1)}{\Gamma(\frac{k}{2}-1)} -1- i^k  e\pb{-\frac{qn}{pr}}.
\end{align}
The integral here is evaluated as a sum using Lemma \ref{lem:keyintegral}:
\begin{align}
\label{eq:evaluated} \frac{1}{2\pi i} \int_{(\frac34)} \Gamma(w) e^{i \frac{\pi}{2} w} \pb{\frac{2\pi n q}{pr}}^{w} \frac{\Gamma(\frac{k}{2}-w)}{\Gamma(\frac{k}{2}+w)}dw =\sum_{j=0}^{\frac{k}{2}-1}  \pb{\frac{2\pi i q n }{pr}}^{-j} \frac{(\frac{k}{2}-1+j)!}{j!(\frac{k}{2}-1-j)!} \pb{(-1)^j + i^k e\pb{-\frac{qn}{pr}} }. 
\end{align}
The summand corresponding to $j=0$ equals $1 + i^ke(-\frac{qn}{pr})$, which is exactly canceled out when \eqref{eq:evaluated} is inserted in \eqref{eq:Q0identity}. This gives Lemma \ref{lem:Qidentity}.

\section{Proof of Theorem \ref{thm:mainthm2}}

The proof of Theorem \ref{thm:mainthm2} is similar to the proof of Theorem \ref{thm:mainthm}, and in fact simpler, so we only sketch the argument. The proof begins with the same set-up and initial contour shifting step, yielding an identity just like \eqref{eq:Mfsprq4}. The left hand side of this identity gives (at $s=0$) two of the three main terms required on the left hand side of \eqref{eq:thm2eq}. The right hand side of the identity contains an additively twisted $L$-series and an integral transform. The $L$-series at $s=0$ is seen to contribute to the the right hand side of \eqref{eq:thm2eq} after applying the functional equation, which in the Maass case can be found in \cite[Section A.4]{KMV}. Thus it remains to analytically continue the integral transform to $s=0$ to obtain the third main term on the left hand side of \eqref{eq:thm2eq}, up to lower order terms on the right hand side of \eqref{eq:thm2eq}. As before, the functional equation is applied to the integrand of this transform. The transform then looks similar to \eqref{eq:Js2}, except that the ratio of gamma factors is
\[
\frac{\Gamma(\frac14+ \frac{\kappa_g}{2} + \frac{it_g}{2} -\frac{s}{2}-\frac{w}{2}  )\Gamma(\frac14 + \frac{\kappa_g}{2}  - \frac{it_g}{2}  -\frac{s}{2}-\frac{w}{2} ) }{  \Gamma(\frac14 + \frac{\kappa_g}{2}  + \frac{it_g}{2} + \frac{s}{2}+\frac{w}{2}  )\Gamma(\frac14 + \frac{\kappa_g}{2}  - \frac{it_g}{2} +\frac{s}{2}+\frac{w}{2} )  }, 
\]
where $t_g$ is the spectral parameter of $g$ and $\kappa_g=\pm 1$ depending on the parity of $g$. The proof now follows a different path because we do not know how to evaluate the transform in an exact way as we did (at $s=0$) in the holomorphic case using Lemma \ref{lem:keyintegral}. The strategy is to evaluate the transform as an asymptotic series up to an arbitrarily small error term using a stationary phase expansion. This follows the same method, sketched below, from \cite{bkt}, which the reader can consult for more details. First we shift the line of integration of the transform far to the left to $\Re(w)=-c$ for any large positive half-integer $c$, crossing poles at $w\in\mathbb{Z}_{<0}$ whose residues contribute to the right hand side of \eqref{eq:thm2eq}. The shifted integral is studied in dyadic intervals, as in Lemma \ref{lem:statphase}. For $|\Im(w)|<1$ the integral is bounded absolutely to get a contribution of $O((\frac{q}{pr})^{-c})$. For $|\Im(w)|\ge 1$, we apply Stirling's expansion. Although the ratio of gamma factors is different, the oscillatory behavior is the same as in \eqref{eq:stirling-expansion}. In segments away from the stationary points, we apply integration by parts repeatedly, obtaining a contribution of $O((\frac{q}{pr})^{-c})$. Close to the stationary points, we apply a stationary phase expansion just as in \eqref{eq:stat-expansion}. The main term which arises combines, using Lemma \ref{lem:recipexp}, with the additive character just as in \eqref{eq:J-rewrite} to give the third required main term on the left hand side of \eqref{eq:thm2eq}. The lower order main terms from the Stirling and stationary phase expansions contribute to the right hand side of \eqref{eq:thm2eq}, while the remainders contribute $O((\frac{q}{pr})^{-c})$ on taking sufficiently long expansions.

\bibliographystyle{amsalpha}
\bibliography{refs}

\end{document}